\documentclass[oneside,10pt]{article}
\usepackage[b5paper]{geometry}	

\usepackage{amsfonts,amsmath,latexsym,amssymb}
\usepackage{mathrsfs,upref}
\usepackage{mathptmx}		    	                            		
\usepackage{amscd}
\usepackage{amsthm}
\usepackage{graphicx}
\usepackage{amsmath}
\usepackage{lipsum}

\newcommand\blfootnote[1]{%
  \begingroup
  \renewcommand\thefootnote{}\footnote{#1}%
  \addtocounter{footnote}{-1}%
  \endgroup}

\newtheorem{theorem}{Theorem}

\newtheorem{corollary}[theorem]{Corollary}

\newtheorem{definition}[theorem]{Definition}

\newtheorem{lemma}[theorem]{Lemma}

\newtheorem{proposition}[theorem]{Proposition}
\newtheorem{remark}[theorem]{Remark}

\newcommand\supp{\mathop{\rm supp}}
\newcommand\esssup{\mathop{\rm ess \, sup}}
\newcommand\essinf{\mathop{\rm ess \, inf}}

\begin{document}

\title{A note on variable (Hardy-)Lorentz spaces}
\author{Pablo Rocha}

\maketitle

\begin{abstract}
The purpose of this note is to establish further properties of the variable Lorentz spaces $\mathfrak{L}^{p(\cdot), q(\cdot)}(\mathbb{R}^n)$ introduced by L. Ephremidze, V. Kokilashvili and S. Samko, which will allow us to apply the theory of Hardy spaces associated with ball quasi-Banach function spaces, and so define the variable Hardy-Lorentz spaces associated with $\mathfrak{L}^{p(\cdot), q(\cdot)}(\mathbb{R}^n)$. Then, the finite and infinite atomic decompositions for these spaces will be deduced immediately. We also obtain the boundedness of singular integrals and fractional type operators in variable Hardy-Lorentz spaces, and provide a Fefferman-Stein vector-valued inequality for the fractional maximal operator on the $r$-convexification of variable Lorentz spaces. 
\end{abstract}


\section{Introduction}

L. Ephremidze, V. Kokilashvili and S. Samko in \cite{Ephre} introduced the variable Lorentz spaces 
$\mathfrak{L}^{p(\cdot), q(\cdot)}(\Omega)$, where $\Omega \subseteq \mathbb{R}^n$ is an open set and 
$p(\cdot), q(\cdot) : (0, |\Omega|) \to (0, \infty)$ are certain variable exponents. When, $p(\cdot) = p$ and $q(\cdot) = q$ 
are constants, the spaces $\mathfrak{L}^{p(\cdot), q(\cdot)}(\Omega)$ coincide with the classical Lorentz spaces $L^{p, q}(\Omega)$.
However, the spaces $\mathfrak{L}^{p(\cdot), p(\cdot)}(\Omega)$ might not coincide with the variable Lebesgue spaces 
$L^{p(\cdot)}(\Omega)$, since in the definition of these spaces the exponent $p(\cdot)$ is defined on whole $\Omega$. 
In \cite{Ephre} the authors, assuming decay conditions of log-type on the exponents $p(t)$ and $q(t)$ as $t \to 0^{+}$ and $t \to \infty$, obtained the boundedness of singular integrals and fractional type operators, and corresponding ergodic operators in the spaces 
$\mathfrak{L}^{p(\cdot), q(\cdot)}(\Omega)$ (see also \cite{Diening2}). On the other hand, H. Kempka and J. Vyb\'iral in \cite{Kempka} introduced another variant of variable Lorentz spaces $L_{p(\cdot), q(\cdot)}(\mathbb{R}^n)$, where $p(\cdot), q(\cdot) : \mathbb{R}^n \to (0, \infty)$ are certain variable exponents. They investigated their basic properties, among them, proved the identity 
$L_{p(\cdot), p(\cdot)}(\mathbb{R}^n) = L^{p(\cdot)}(\mathbb{R}^n)$. They also showed that 
$\left(L^{p(\cdot)}(\mathbb{R}^n), \, L^{\infty}(\mathbb{R}^n) \right)_{\theta, q} = L_{\widetilde{p}(\cdot), q}(\mathbb{R}^n)$, 
where $q$ is a positive constant, $0 < \theta < 1$ and $\frac{1}{\widetilde{p}(\cdot)} = \frac{1-\theta}{p(\cdot)}$ (i.e.: the spaces 
$L_{p(\cdot), q}$ arise through real interpolation between $L^{p(\cdot)}$ and $L^{\infty}$). Moreover, they answered in a
negative way the Question 2.8 posed in \cite{Diening}. This question concerns validity of Marcinkiewicz interpolation theorem in the variable context.

Later, V. Almeida, J. J. Betancor and L. Rodr\'iguez-Mesa in \cite{Almeida} introduced anisotropic Hardy–Lorentz spaces with variable exponents $H^{p(\cdot), q(\cdot)}(\mathbb{R}^n, A)$ associated with dilations $A$ in $\mathbb{R}^n$ and modelled on $\mathfrak{L}^{p(\cdot), q(\cdot)}(\mathbb{R}^n)$. They also established maximal characterizations and atomic decompositions for these variable anisotropic Hardy–Lorentz spaces.

Recently, Y. Jiao, Y. Zuo, D. Zhou and L. Wu in \cite{Jiao} presented another kind of variable Hardy-Lorentz spaces 
$H^{p(\cdot), q}(\mathbb{R}^n)$ fashioned on $L_{p(\cdot), q}(\mathbb{R}^n)$, where $q$ is a positive constant. They obtained an atomic decomposition for elements of $H^{p(\cdot), q}(\mathbb{R}^n)$, and developed a theory of real interpolation and formulate the dual space of the variable Hardy–Lorentz space for $0 < p_{-} \leq p_{+} \leq 1$ and $0 < q < \infty$. They also investigated the boundedness of singular operators on $H^{p(\cdot), q}(\mathbb{R}^n)$ for $0 < p_{-} \leq p_{+} < \infty$ and $0 < q < \infty$.

The purpose of this work is to incorporate to the spaces $\mathfrak{L}^{p(\cdot), q(\cdot)}(\mathbb{R}^n)$ within the framework developed by Y. Sawano, K.-P. Ho, D. Yang and S. Yang in \cite{Sawano}, and so define and study the variable Hardy-Lorentz spaces associated with $\mathfrak{L}^{p(\cdot), q(\cdot)}(\mathbb{R}^n)$ from this broader perspective.

\

The paper is organized as follows.

In Section 2, we present the variable Lorentz spaces $\mathfrak{L}^{p(\cdot), q(\cdot)}(\mathbb{R}^n)$ following to \cite{Ephre}, and provide 
some further properties of these spaces. For instance, we will show that the quasi-norm of 
$\mathfrak{L}^{p(\cdot), q(\cdot)}(\mathbb{R}^n)$ is absolutely continuous when $q(0) \leq p(0)$ and $p(\infty) \leq q(\infty)$, a H\"older's inequality will be established for these spaces and will also prove that the K\"othe dual $(\mathfrak{L}^{p(\cdot), q(\cdot)}(\mathbb{R}^n))' = \mathfrak{L}^{p'(\cdot), q'(\cdot)}(\mathbb{R}^n)$. This identity was established without proof in \cite[Lemma 2.7]{Ephre}, and subsequently used in \cite{Almeida} and \cite{Zhao}. Assuming $q(0) \leq p(0)$ and $p(\infty) \leq q(\infty)$, we will show that 
$(\mathfrak{L}^{p(\cdot), q(\cdot)}(\mathbb{R}^n))'$ is isomorphic to the dual space $(\mathfrak{L}^{p(\cdot), q(\cdot)}(\mathbb{R}^n))^{*}$. Section 2 concludes with a Fefferman-Stein valued-vector inequality for the Hardy-Littlewood maximal operator.

In Section 3, we define the variable Hardy-Lorentz spaces $H^{p(\cdot), q(\cdot)}(\mathbb{R}^n)$ associated with 
$\mathfrak{L}^{p(\cdot), q(\cdot)}(\mathbb{R}^n)$. By means of the properties described in Section 2 and the theories developed by 
X. Yan, D. Yang and W. Yuan in \cite{Yan}, and by Y. Sawano et al. in \cite{Sawano}, we obtain the finite and infinite atomic decompositions for the spaces $H^{p(\cdot), q(\cdot)}(\mathbb{R}^n)$ respectively.

Finally, in Section 4 we obtain the boundedness of singular integrals and fractional type operators in $H^{p(\cdot), q(\cdot)}(\mathbb{R}^n)$, as well as a Fefferman-Stein vector-valued inequality for the fractional maximal operator on the $r$-convexification of 
$\mathfrak{L}^{p(\cdot), q(\cdot)}(\mathbb{R}^n)$. These results are based on two previous works of the present author, which were realized within the frame of the ball quasi-Banach function spaces (see \cite{Rocha} and \cite{Rocha2}).

\

\textbf{Notation:} The symbol $S \lesssim T$ stands for the inequality $S \leq cT$ for some constant $c$. The symbol $S \approx T$ 
stands for $T \lesssim S \lesssim T$. We denote by $Q(x_0, r)$ the cube centered at $x_0 \in \mathbb{R}^{n}$ with side lenght $r$. 
Given $\gamma >0$ and a cube $Q = Q(x_0, r)$, we set $\gamma Q = Q(x_0, \gamma r)$. Denote by $\mathcal{Q}$ the set of all cubes 
having their edges parallel to the coordinate axes. The orthogonal group $\mathcal{O}(n)$ is defined by 
$\mathcal{O}(n):=\{ A \in GL_n(\mathbb{R}) : A^{t}= A^{-1} \}$. The group $\mathcal{O}(n)$ induces an action on functions by 
$f_{A}(x) = f(A^{-1}x)$, where $A \in \mathcal{O}(n)$. For a measurable subset $E \subset \mathbb{R}^{n}$ we denote 
$|E|$ and $\chi_E$ the Lebesgue measure of $E$ and the characteristic function of $E$ respectively. 
Given a real number $s \geq 0$, we write $\lfloor s \rfloor$ for the integer part of $s$. As usual we denote with 
$\mathcal{S}(\mathbb{R}^{n})$ the space of smooth and rapidly decreasing functions, with 
$\mathcal{S}'(\mathbb{R}^{n})$  the dual space. If $\beta$ is the multiindex $\beta=(\beta_1, ..., \beta_n)$, then 
$|\beta| = \beta_1 + ... + \beta_n$. Given a function $g$ on $\mathbb{R}^n$ and $t > 0$, we write $g_t(x) = t^{-n} g(t^{-1} x)$. Let $d$ be a non negative integer and $\mathcal{P}_{d}$ the subspace of $L^{1}_{loc}(\mathbb{R}^{n})$ formed by all the polynomials of degree at most $d$. Given a measurable function $h$, the expression $h \perp \mathcal{P}_{d}$ stands for $\int h(x) P(x) dx = 0$ for all $P \in \mathcal{P}_{d}$.

Throughout this paper, $C$ will denote a positive constant, not necessarily the same at each occurrence.

\section{Variable Lorentz spaces}

Now, we present the variable Lorentz spaces following to \cite{Ephre}. A measurable function $p(\cdot) : (0, \infty) \to (0, \infty)$ is called a variable exponent on $(0, \infty)$. Denote by $\mathcal{P}(0, \infty)$ the collection of all variable exponents $p(\cdot)$ on 
$(0, \infty)$. Let $\mathcal{P}_{0}(0, \infty)$ be the collection of all $p(\cdot) \in \mathcal{P}(0, \infty)$ satisfying
\[
0 < p_{-} := \essinf_{t \in (0, \infty)} p(t) \leq \esssup_{t \in (0, \infty)} p(t) =: p_{+} < \infty.
\]
If $p_{-} > 1$, define the conjugate exponent $p'(\cdot)$ by $\frac{1}{p'(t)} := 1 - \frac{1}{p(t)}$.
 
Given $0 \leq a < \infty$, we set
\[
\mathcal{P}_{a}(0,\infty) = \{ p(\cdot) \in \mathcal{P}(0, \infty) : a < p_{-} \leq p_{+} < \infty \}.
\]
We are interested in the special cases $a = 0$, or $a = 1$.

\begin{definition}
By $\mathbb{P}_0(0, \infty)$ we denote the class of variable exponents $p(\cdot) \in \mathcal{P}_{0}(0, \infty)$ such that there exist the limits
\[
p(0) := \lim_{t \to 0^{+}} p(t) \,\,\,\, \text{and} \,\,\,\, p(\infty) := \lim_{t \to \infty} p(t),
\]
and the conditions
\[
|p(t) - p(0)| \leq \frac{C}{-\log(t)}, \,\, 0 < t < \frac{1}{2}, \,\,\,\,\,\, \text{and} \,\,\,\,\,\, |p(t) - p(\infty)| \leq 
\frac{C}{\log(e + t)}, \,\, 0 < t < \infty
\]
are satisfied for some positive constant $C$ independent of $t$. Given $0 < a < \infty$, we also denote
\[
\mathbb{P}_{a}(0, \infty) = \mathbb{P}_0(0, \infty) \cap \mathcal{P}_{a}(0, \infty). 
\]
\end{definition}

The non-increasing rearrangement of a measurable function $f$ on $\mathbb{R}^n$ is defined by
\[
f^{*}(t) = \inf\left\{ s \geq 0 : |\{ x \in \mathbb{R}^n : |f(x)| > s \}| \leq t \right\}, \,\,\, \text{for all} \,\,\, 0 < t < \infty.
\]
Some elementary properties of the function $f^{*}$ are (see \cite{Grafakos})
\[
f^{*}(t) = |f|^{*}(t), \,\,\,\, \text{for all} \,\,\, 0 < t < \infty,
\]
\[
(\lambda f)^{*}(t) = |\lambda| f^{*}(t), \,\,\,\, \text{for} \,\,\, \lambda \in \mathbb{C} \,\,\, \text{and all} \,\,\, 0 < t < \infty,
\]
\[
(f + g)^{*}(t) \leq f^{*}(t/2) + g^{*}(t/2), \,\,\,\, \text{for all} \,\,\, 0 < t < \infty,
\]
\[
|g(x)| \leq |f(x)| \,\, \text{a.e.}x \in \mathbb{R}^n \,\,\, \text{implies that} \,\,\, g^{*}(t) \leq f^{*}(t), \,\,\,\, \text{for all} 
\,\,\, 0 < t < \infty,
\]
\[
|\{ t \in (0, \infty) : f^{*}(t) > s \}| = |\{ x \in \mathbb{R}^n : |f(x)| > s \}|, \,\,\,\, \text{for all} \,\,\, s > 0,
\]
and
\[
(|f|^s)^{*} \leq (f^{*})^s, \,\,\,\, \text{when} \,\,\, 0 < s < \infty.
\]

\begin{definition} \label{var Lorentz}
Let $p(\cdot)$ and $q(\cdot) \in \mathcal{P}_{0}(0, \infty)$. By $\mathfrak{L}^{p(\cdot), q(\cdot)}(\mathbb{R}^n)$ we denote the space of all measurable functions $f$ on $\mathbb{R}^n$ such that
\[
t^{\frac{1}{p(t)} - \frac{1}{q(t)}} f^{*}(t) \in L^{q(\cdot)}(0, \infty),
\]
that is 
\[
\rho_{p(\cdot), q(\cdot)}(f):= \int_{0}^{\infty} t^{\frac{q(\cdot)}{p(\cdot)} - 1} [f^{*}(t)]^{q(t)} dt < \infty,
\]
and, for every $f \in \mathfrak{L}^{p(\cdot), q(\cdot)}(\mathbb{R}^n)$, we consider the Luxemburg norm
\[
\| f \|_{\mathfrak{L}^{p(\cdot), q(\cdot)}} = \inf \left\{ \lambda > 0 : \rho_{p(\cdot), q(\cdot)}(f/\lambda) \leq 1 \right\} =
\left\|  t^{\frac{1}{p(\cdot)} - \frac{1}{q(\cdot)}} f^{*}(t) \right\|_{L^{q(\cdot)}(0, \infty)}.
\]
The space $\mathfrak{L}^{p(\cdot), q(\cdot)}(\mathbb{R}^n)$ is called the variable Lorentz space.
\end{definition}

\begin{remark} \label{equivalente}
In the case $p(\cdot)$ and $q(\cdot) \in \mathbb{P}_{0}(0, \infty)$, we have that
\[
\int_{0}^{\infty} t^{\frac{q(\cdot)}{p(\cdot)} - 1} [f^{*}(t)]^{q(t)} dt \approx \int_{0}^{1/2} t^{\frac{q(0)}{p(0)} - 1} [f^{*}(t)]^{q(t)} dt + \int_{1/2}^{\infty} t^{\frac{q(\infty)}{p(\infty)} - 1} [f^{*}(t)]^{q(t)} dt,
\]
where the implicit constants do not depend on $f$.
\end{remark}

\begin{lemma} \label{Lpqs}
Let $p(\cdot)$ and $q(\cdot) \in \mathcal{P}_{0}(0, \infty)$. Then, for any $s \in (0, \infty)$ and any measurable function $f$, one has
\[
\| f \|_{\mathfrak{L}^{p(\cdot), q(\cdot)}}^{s} = \| |f|^s \|_{\mathfrak{L}^{\frac{p(\cdot)}{s}, \frac{q(\cdot)}{s}}}.
\]
\end{lemma}

\begin{proof}
Fix $s \in (0, \infty)$. Then, by definition and the properties of $f^{*}$
\[
\| |f|^s \|_{\mathfrak{L}^{p(\cdot)/s, q(\cdot)/s}} = \inf \left\{ \lambda > 0 : 
\int_{0}^{\infty}  t^{\frac{q(\cdot)}{p(\cdot)} - 1} [f^{*}(t)/\lambda^{1/s}]^{q(t)} dt \leq 1 \right\} 
\]
\[
= \inf \left\{ \mu^s > 0 : \int_{0}^{\infty}  t^{\frac{q(\cdot)}{p(\cdot)} - 1} [f^{*}(t)/\mu]^{q(t)} dt \leq 1 \right\} = 
\| f \|_{\mathfrak{L}^{p(\cdot), q(\cdot)}}^{s}.
\]
Then, the lemma follows.
\end{proof}

\begin{lemma} \label{converg a cero}
Let $p(\cdot)$ and $q(\cdot) \in \mathcal{P}_{0}(0, \infty)$. If $\{ f_j \}$ is a sequence of measurable functions on 
$\mathbb{R}^n$ such that $\rho_{p(\cdot), q(\cdot)}(f_j) \to 0$, then $\| f_j \|_{\mathfrak{L}^{p(\cdot), q(\cdot)}} \to 0$.
\end{lemma}

\begin{proof}
Suppose that $0 < \rho_{p(\cdot), q(\cdot)}(f_j) \to 0$. Given $0 < \epsilon < 1$, for all $j$ large enough, we have 
$0 < \rho_{p(\cdot), q(\cdot)}(f_j) < \epsilon^{q_{+}}$ and
\[
\rho_{p(\cdot), q(\cdot)} \left( \rho_{p(\cdot), q(\cdot)}(f_j)^{-1/q_{+}} f_j \right)  \leq 
\rho_{p(\cdot), q(\cdot)}(f_j)^{-1} \rho_{p(\cdot), q(\cdot)}(f_j) = 1, 
\]
so $\| f_j \|_{\mathfrak{L}^{p(\cdot), q(\cdot)}} \leq \rho_{p(\cdot), q(\cdot)}(f_j)^{1/q_{+}} < \epsilon$. Then, 
$\| f_j \|_{\mathfrak{L}^{p(\cdot), q(\cdot)}} \to 0$.
\end{proof}

\begin{theorem} (H\"older's inequality) \label{Holder}
Let $p(\cdot)$ and $q(\cdot) \in \mathcal{P}_{1}(0, \infty)$. Then
\begin{equation} \label{Holder ineq}
\int_{\mathbb{R}^n} |f(x) g(x)| dx \leq \left( 1 + \frac{1}{q_{-}} - \frac{1}{q_{+}} \right)
\| f \|_{\mathfrak{L}^{p'(\cdot), q'(\cdot)}(\mathbb{R}^n)} \| g \|_{\mathfrak{L}^{p(\cdot), q(\cdot)}(\mathbb{R}^n)}.
\end{equation}
\end{theorem}

\begin{proof}
By \cite[Theorem 2.2 on p. 44]{Bennett}, it follows that
\begin{equation} \label{Hardy}
\int_{\mathbb{R}^n} |f(x) g(x)| dx \leq \int_{0}^{\infty} f^{*}(t) g^{*}(t) dt = \int_{0}^{\infty} t^{1/p'(t)} f^{*}(t) t^{1/p(t)}g^{*}(t) 
\frac{dt}{t}.
\end{equation}
Now, the Young's inequality, i.e. $a b \leq \displaystyle{\frac{a^{q'(t)}}{q'(t)} + \frac{b^{q(t)}}{q(t)}}$, gives
\begin{equation} \label{Young}
t^{1/p'(t)} f^{*}(t) t^{1/p(t)}g^{*}(t) \leq \frac{t^{q'(t)/p'(t)}}{q'(t)} [f^{*}(t)]^{q'(t)} + \frac{t^{q(t)/p(t)}}{q(t)} [g^{*}(t)]^{q(t)}.
\end{equation}
Then, from (\ref{Hardy}) and (\ref{Young}) we obtain
\[
\int_{\mathbb{R}^n} |f(x) g(x)| dx \leq \frac{\rho_{p'(\cdot), q'(\cdot)}(f)}{(q')_{-}} + \frac{\rho_{p(\cdot), q(\cdot)}(g)}{q_{-}},
\]
by considering $\rho_{p'(\cdot), q'(\cdot)}(f), \, \rho_{p(\cdot), q(\cdot)}(g) \leq 1$ and since $\frac{1}{(q')_{-}} = \frac{1}{(q_{+})'} = 1 - \frac{1}{q_{+}}$, we get
\[
\int_{\mathbb{R}^n} |f(x) g(x)| dx \leq 1 + \frac{1}{q_{-}} - \frac{1}{q_{+}},
\]
which leads to (\ref{Holder ineq}).
\end{proof}

Let
\[
f^{**}(t) := \frac{1}{t} \int_{0}^{t} f^{*}(s) ds.
\]
Given $p(\cdot)$ and $q(\cdot) \in \mathcal{P}_{0}(0, \infty)$, we define the space $\mathfrak{L}_{p(\cdot), q(\cdot)}(\mathbb{R}^n)$ as the collection of all measurable functions $f$ on $\mathbb{R}^n$ such that
\[
\| f \|_{\mathfrak{L}_{p(\cdot), q(\cdot)}} := \left\| t^{\frac{1}{p(t)} - \frac{1}{q(t)}}  f^{**}(t) \right\|_{L^{q(\cdot)}(0, \infty)} < \infty.
\]
By \cite[Theorem 2.8]{Ephre}, the couple $\left( \mathfrak{L}_{p(\cdot), q(\cdot)}(\mathbb{R}^n), \| \cdot \|_{\mathfrak{L}_{p(\cdot), q(\cdot)}} \right)$ results a Banach function space when $p(\cdot)$ and $q(\cdot) \in \mathbb{P}_{1}(0, \infty)$.

\begin{definition} \label{r-convexification}
Let $X$ be a (quasi-)Banach function space and $r \in (0, \infty)$. The $r$-convexification $X^r$ of $X$ is defined by setting 
$X^r := \{ f \in \mathfrak{M} : |f|^r \in X \}$ equipped with the quasi-norm $\| f \|_{X^r} : =  \| |f|^r \|^{1/r}_{X}$.
\end{definition}

For $X = \mathfrak{L}_{p(\cdot), q(\cdot)}(\mathbb{R}^n)$ and $s \in (0, \infty)$, we have that 
$X^{\frac{1}{s}} = \mathfrak{L}_{\frac{p(\cdot)}{s}, \frac{q(\cdot)}{s}}(\mathbb{R}^n)$ and 
$\| f \|_{X^{\frac{1}{s}}} = \| f \|_{\mathfrak{L}_{\frac{p(\cdot)}{s}, \frac{q(\cdot)}{s}}}$. 

\begin{proposition} \label{s-convex}
Let $p(\cdot)$ and $q(\cdot) \in \mathbb{P}_{0}(0, \infty)$, then for any $s \in (0, \min\{1, p_{-}, q_{-} \})$ and any sequence 
$\{ f_j \}_{j=1}^{\infty} \subset (\mathfrak{L}_{p(\cdot), q(\cdot)}(\mathbb{R}^n))^{\frac{1}{s}}$,
\[
\left\| \sum_{j=1}^{\infty} |f_j| \right\|_{(\mathfrak{L}_{p(\cdot), q(\cdot)})^{\frac{1}{s}}} \leq 
\sum_{j=1}^{\infty} \left\| f_j \right\|_{(\mathfrak{L}_{p(\cdot), q(\cdot)})^{\frac{1}{s}}}.
\]
In others words, $\mathfrak{L}_{p(\cdot), q(\cdot)}(\mathbb{R}^n)$ is strictly $s$-convex.
\end{proposition}

\begin{proof}
Follows from that $(\mathfrak{L}_{p(\cdot), q(\cdot)}(\mathbb{R}^n))^{\frac{1}{s}} = 
\mathfrak{L}_{\frac{p(\cdot)}{s}, \frac{q(\cdot)}{s}}(\mathbb{R}^n)$ is a Banach function space and 
\cite[Lemma 1.5 (i) on p. 4]{Bennett}.
\end{proof}

\begin{proposition} \label{norm Lpq}
Let $p(\cdot)$ and $q(\cdot) \in \mathbb{P}_{1}(0, \infty)$, then $\mathfrak{L}^{p(\cdot), q(\cdot)}(\mathbb{R}^n) = 
\mathfrak{L}_{p(\cdot), q(\cdot)}(\mathbb{R}^n)$, and there exists a constant $C > 1$ such that for any 
$f \in \mathfrak{L}^{p(\cdot), q(\cdot)}(\mathbb{R}^n)$
\[
\| f \|_{\mathfrak{L}^{p(\cdot), q(\cdot)}} \leq \| f \|_{\mathfrak{L}_{p(\cdot), q(\cdot)}} \leq C \| f \|_{\mathfrak{L}^{p(\cdot), q(\cdot)}}.
\]
\end{proposition}

\begin{proof}
The first inequality follows from that $f^{*} \leq f^{**}$, and the second one from \cite[Theorem 2.4]{Ephre}.
\end{proof}

From the definition of the space $\mathfrak{L}^{p(\cdot), q(\cdot)}(\mathbb{R}^n)$ and the subsequent results, one can check the following properties.

\begin{theorem} \label{Lpq quasi}
Let $p(\cdot)$ and $q(\cdot) \in \mathbb{P}_{0}(0, \infty)$, then the space $\mathfrak{L}^{p(\cdot), q(\cdot)}(\mathbb{R}^n)$ together with the Luxemburg norm $\| \cdot \|_{\mathfrak{L}^{p(\cdot), q(\cdot)}}$ results a quasi-Banach function space. That is,

$(i)$ $\| f \|_{\mathfrak{L}^{p(\cdot), q(\cdot)}} = 0$ implies that $f = 0$ a.e.;

$(ii)$ $\| \lambda f \|_{\mathfrak{L}^{p(\cdot), q(\cdot)}} = |\lambda| \| f \|_{\mathfrak{L}^{p(\cdot), q(\cdot)}}$ for all 
$\lambda \in \mathbb{C}$ and all $f \in \mathfrak{L}^{p(\cdot), q(\cdot)}(\mathbb{R}^n)$;

$(iii)$ there exists $C \geq 1$ such that $\| f+g \|_{\mathfrak{L}^{p(\cdot), q(\cdot)}} \leq C (\| f \|_{\mathfrak{L}^{p(\cdot), q(\cdot)}} + \| g \|_{\mathfrak{L}^{p(\cdot), q(\cdot)}})$ for all $f, g \in \mathfrak{L}^{p(\cdot), q(\cdot)}(\mathbb{R}^n)$;

$(iv)$ if $f$ is a measurable function and $g \in \mathfrak{L}^{p(\cdot), q(\cdot)}(\mathbb{R}^n)$ are such that $|f| \leq |g|$ a.e., 
then $f \in \mathfrak{L}^{p(\cdot), q(\cdot)}(\mathbb{R}^n)$ and 
$\| f \|_{\mathfrak{L}^{p(\cdot), q(\cdot)}} \leq \| g \|_{\mathfrak{L}^{p(\cdot), q(\cdot)}}$ (in particular 
$\| f \|_{\mathfrak{L}^{p(\cdot), q(\cdot)}} = \| |f| \|_{\mathfrak{L}^{p(\cdot), q(\cdot)}}$);

$(v)$ if $0 \leq f_n \uparrow f$ a.e., then either $f \notin \mathfrak{L}^{p(\cdot), q(\cdot)}(\mathbb{R}^n)$ and 
$\| f_n \|_{\mathfrak{L}^{p(\cdot), q(\cdot)}} \uparrow \infty$, or $f \in \mathfrak{L}^{p(\cdot), q(\cdot)}(\mathbb{R}^n)$ and 
$\| f_n \|_{\mathfrak{L}^{p(\cdot), q(\cdot)}} \uparrow \| f \|_{\mathfrak{L}^{p(\cdot), q(\cdot)}}$;

$(vi)$ $\| \chi_E \|_{\mathfrak{L}^{p(\cdot), q(\cdot)}} < \infty$ for all measurable set $E$ of $\mathbb{R}^n$ with $|E| < \infty$.
\end{theorem}

In particular, the space $\mathfrak{L}^{p(\cdot), q(\cdot)}(\mathbb{R}^n)$ is a ball quasi-Banach function space (see \cite[Definition 2.2]{Sawano}).

The following result shows that the quasi-norm $\| \cdot \|_{\mathfrak{L}^{p(\cdot), q(\cdot)}}$ is absolutely continuous when
$q(0) \leq p(0)$ and $p(\infty) \leq q(\infty)$.

\begin{proposition} \label{abs cont}
Let $p(\cdot)$ and $q(\cdot) \in \mathbb{P}_{0}(0, \infty)$ such that $q(0) \leq p(0)$ and $p(\infty) \leq q(\infty)$. Then, for 
every $f \in \mathfrak{L}^{p(\cdot), q(\cdot)}(\mathbb{R}^n)$, 
\[
\| f \chi_{E_j} \|_{\mathfrak{L}^{p(\cdot), q(\cdot)}} \downarrow 0
\]
whenever $\{ E_j \}_{j=1}^{\infty}$ is a sequence of measurable 
sets of $\mathbb{R}^n$ that satisfies $E_j \supset E_{j+1}$ for all $j \in \mathbb{N}$ and $\bigcap_{j=1}^{\infty} E_j = \emptyset$.
\end{proposition}

\begin{proof}
Given $f \in \mathfrak{L}^{p(\cdot), q(\cdot)}(\mathbb{R}^n)$, by Remark \ref{equivalente} and since $q(0) \leq p(0)$ and 
$p(\infty) \leq q(\infty)$, to apply conveniently Chebyshev's inequality we have that
\[
|\{ t \in (0, \infty) : f^{*}(t) > s \}| < \infty, \,\,\,\, \text{for all} \,\, s > 0.
\]
As $|\{ t \in (0, \infty) : f^{*}(t) > s \}| = |\{ x \in \mathbb{R}^n : |f(x)| > s \}|$,
\[
|\{ x \in \mathbb{R}^n : |f(x)| > s \}| < \infty, \,\,\,\, \text{for all} \,\, s > 0.
\]
So, for any fixed $s > 0$
\[
|\{ x \in \mathbb{R}^n : |(f \chi_{E_j})(x)| > s \} = |\{ x \in \mathbb{R}^n : |f(x)| > s \} \cap E_j | \downarrow 0.
\]
Then, for every $n \in \mathbb{N}$, there exists $j_0(n)$ such that
\[
|\{ x \in \mathbb{R}^n : |(f \chi_{E_j})(x)| > 2^{-n} \} \leq 2^{-n}, \,\,\,\, \text{for all} \,\, j \geq j_0(n),
\]
which leads to $(f \chi_{E_j})^{*}(t) \leq 2^{-n}$, for all $t \geq 2^{-n}$ and $j \geq j_0(n)$. Thus, for any $t > 0$, 
$(f \chi_{E_j})^{*}(t) \downarrow 0$. Since $(f \chi_{E_j})^{*} \leq f^{*}$ and $f \in \mathfrak{L}^{p(\cdot), q(\cdot)}$, by 
dominated convergence Theorem, we obtain $\rho_{p(\cdot), q(\cdot)}(f \chi_{Ej}) \downarrow 0$. Finally, the proposition follows from 
Lemma \ref{converg a cero}.
\end{proof}

\begin{remark}
We observe that the notion of absolutely continuous (quasi-)norm presented in \cite[Proposition 3.1]{Almeida} is not correct. There, they assume $|E_1| < \infty$, which is not required in the definition of absolutely continuous (quasi-)norm (see e.g. \cite[Definition 3.2]{Wang}). For instance, \cite[Proposition 3.1]{Almeida} does not cover the case $E_j = \{ x \in \mathbb{R}^n : |x| > j \}$, where $j \in \mathbb{N}$.
\end{remark}

We point out that \cite[Theorem 3.11]{Bennett} also applies to quasi-Banach function spaces. So, from Proposition \ref{abs cont}, it obtains the following result.

\begin{corollary} \label{simple funct}
Let $p(\cdot)$ and $q(\cdot) \in \mathbb{P}_{0}(0, \infty)$ such that $q(0) \leq p(0)$ and $p(\infty) \leq q(\infty)$. Then, the collection of all simple functions supported in sets of finite measure is dense in $\mathfrak{L}^{p(\cdot), q(\cdot)}(\mathbb{R}^n)$.
\end{corollary}

Let $p(\cdot)$ and $q(\cdot) \in \mathcal{P}_{1}(0, \infty)$, the associate space (or K\"othe dual) 
$(\mathfrak{L}^{p(\cdot), q(\cdot)}(\mathbb{R}^n))'$ is defined by
\[
(\mathfrak{L}^{p(\cdot), q(\cdot)}(\mathbb{R}^n))' := \left\{ f \in \mathfrak{M}(\mathbb{R}^n) : 
\| f \|_{(\mathfrak{L}^{p(\cdot), q(\cdot)}(\mathbb{R}^n))'} < \infty \right\},
\]
where 
\[
\| f \|_{(\mathfrak{L}^{p(\cdot), q(\cdot)}(\mathbb{R}^n))'} := \sup \left\{ \| f g\|_{L^{1}(\mathbb{R}^n)} : 
\| g \|_{\mathfrak{L}^{p(\cdot), q(\cdot)}(\mathbb{R}^n)} \leq 1 \right\}.
\]

We also introduce the following auxiliary space which will be crucial to compute explicitly the K\"othe dual of $\mathfrak{L}^{p(\cdot), q(\cdot)}(\mathbb{R}^n)$.

\begin{definition}
Let $p(\cdot)$ and $q(\cdot) \in \mathcal{P}_{1}(0, \infty)$. By $L^{p(\cdot), q(\cdot)}(0, \infty)$  we denote the space of all measurable functions $h$ on $(0, \infty)$ such that
\[
t^{\frac{1}{p(t)} - \frac{1}{q(t)}} h(t) \in L^{q(\cdot)}(0, \infty),
\]
that is 
\[
\mu_{p(\cdot), q(\cdot)}(h):= \int_{0}^{\infty} t^{\frac{q(\cdot)}{p(\cdot)} - 1} |h(t)|^{q(t)} dt < \infty,
\]
and, for every $h \in L^{p(\cdot), q(\cdot)}(0, \infty)$, we consider the Luxemburg norm
\[
\| h \|_{L^{p(\cdot), q(\cdot)}(0, \infty)} := \inf \left\{ \lambda > 0 : \mu_{p(\cdot), q(\cdot)}(h/\lambda) \leq 1 \right\} 
= \left\|  t^{\frac{1}{p(t)} - \frac{1}{q(t)}} h(t) \right\|_{L^{q(\cdot)}(0, \infty)}.
\] 
\end{definition}

\begin{remark} \label{Luxemburg rep}
Let $p(\cdot)$ and $q(\cdot) \in \mathcal{P}_{1}(0, \infty)$. It is clear that the space $L^{p(\cdot), q(\cdot)}(0, \infty)$, together 
with the norm $\| \cdot \|_{L^{p(\cdot), q(\cdot)}(0, \infty)}$, is a Banach function space. Moreover, for every 
$f \in \mathfrak{L}^{p(\cdot), q(\cdot)}(\mathbb{R}^n)$, we have that 
$\| f \|_{\mathfrak{L}^{p(\cdot), q(\cdot)}(\mathbb{R}^n)} = \| f^{*} \|_{L^{p(\cdot), q(\cdot)}(0, \infty)}$. On the other hand, by 
\cite[Proposition 2.30]{Musilova}, we have for every $g \in (\mathfrak{L}^{p(\cdot), q(\cdot)}(\mathbb{R}^n))'$ that 
$\| g \|_{(\mathfrak{L}^{p(\cdot), q(\cdot)}(\mathbb{R}^n))'} \leq \| g^{*} \|_{(L^{p(\cdot), q(\cdot)}(0, \infty))'}$. For the reverse inequality, once again taking into account \cite[Proposition 2.30]{Musilova}, we proceed as in the proof of the second statement in 
\cite[Theorem 4.10]{Bennett} and obtain $\| g^{*} \|_{(L^{p(\cdot), q(\cdot)}(0, \infty))'} \leq \| g \|_{(\mathfrak{L}^{p(\cdot), q(\cdot)}(\mathbb{R}^n))'}$. Thus, $\| g \|_{(\mathfrak{L}^{p(\cdot), q(\cdot)}(\mathbb{R}^n))'} = 
\| g^{*} \|_{(L^{p(\cdot), q(\cdot)}(0, \infty))'}$, for all $g \in (\mathfrak{L}^{p(\cdot), q(\cdot)}(\mathbb{R}^n))'$.
\end{remark}

\begin{theorem} \label{dual}
Let $p(\cdot)$ and $q(\cdot) \in \mathcal{P}_{1}(0, \infty)$. Then the K\"othe dual 
$\left( \mathfrak{L}^{p(\cdot), q(\cdot)}(\mathbb{R}^n) \right)' = \mathfrak{L}^{p'(\cdot),  q'(\cdot)}(\mathbb{R}^n)$ up to equivalence of quasi-norms.
\end{theorem}

\begin{proof}
Given $f \in \mathfrak{L}^{p'(\cdot),  q'(\cdot)}(\mathbb{R}^n)$, by Theorem \ref{Holder} there exists a constant $C \geq 1$ such that
\[
\int_{\mathbb{R}^n} |f(x) g(x)| dx \leq C
\| f \|_{\mathfrak{L}^{p'(\cdot), q'(\cdot)}(\mathbb{R}^n)} \| g \|_{\mathfrak{L}^{p(\cdot), q(\cdot)}(\mathbb{R}^n)},
\]
for every $g \in \mathfrak{L}^{p(\cdot), q(\cdot)}(\mathbb{R}^n)$. Then,  $\| f \|_{( \mathfrak{L}^{p(\cdot), q(\cdot)}(\mathbb{R}^n))'} \leq C \| f \|_{\mathfrak{L}^{p'(\cdot), q'(\cdot)}(\mathbb{R}^n)}$ for all $f \in \mathfrak{L}^{p'(\cdot), q'(\cdot)}(\mathbb{R}^n)$, and so 
$\mathfrak{L}^{p'(\cdot),  q'(\cdot)}(\mathbb{R}^n) \subset \left( \mathfrak{L}^{p(\cdot), q(\cdot)}(\mathbb{R}^n) \right)'$ embeds continuously.

To see the opposite inclusion, by Remark \ref{Luxemburg rep}, it suffices to prove the following claim: 
\[
\text{If} \,\, \| f \|_{(L^{p(\cdot), q(\cdot)}(0, \infty) )'} \leq 1 \,\,\, \text{and} \,\,\, \mu_{p'(\cdot), q'(\cdot)}(f) < \infty, \,\,\,  \text{then} \,\,\, \| f \|_{L^{p'(\cdot), q'(\cdot)}(0, \infty)} \leq 1.
\]
In fact, given 
$0 \neq f \in (L^{p(\cdot), q(\cdot)}(0, \infty) )'$, let $\{ f_k \}_{k=1}^{\infty}$ be a sequence of simple functions such that 
$0 \leq f_k \uparrow |f|/ \| f \|_{(L^{p(\cdot), q(\cdot)}(0, \infty) )'}$. By the claim above, we have $\| f_k \|_{L^{p'(\cdot), q'(\cdot)}} \leq 1$, which implies that $\mu_{p'(\cdot), q'(\cdot)}(f_k) \leq \| f_k \|_{L^{p'(\cdot), q'(\cdot)}} \leq 1$ for all $k \geq 1$. This inequality leads to
\[
\mu_{p'(\cdot), q'(\cdot)} \left( |f|/\| f \|_{(L^{p(\cdot), q(\cdot)}(0, \infty) )'} \right) \leq 
\liminf_{k \to \infty} \mu_{p'(\cdot), q'(\cdot)}(f_k) \leq 1,
\]
and so 
\begin{equation} \label{Lpq ineq}
\| f \|_{L^{p'(\cdot), q'(\cdot)}(0, \infty)} \leq \| f \|_{(L^{p(\cdot), q(\cdot)}(0, \infty))'}
\end{equation}
for all $f \in (L^{p(\cdot), q(\cdot)}(0, \infty))'$. Then, given $g \in (\mathfrak{L}^{p(\cdot), q(\cdot)}(\mathbb{R}^n))'$, 
by Remark \ref{Luxemburg rep} and applying (\ref{Lpq ineq}) with $f = g^{*}$, we obtain that
\[
\| g \|_{\mathfrak{L}^{p'(\cdot), q'(\cdot)}(\mathbb{R}^n)} \leq \| g \|_{(\mathfrak{L}^{p(\cdot), q(\cdot)}(\mathbb{R}^n))'},
\]
and $\left( \mathfrak{L}^{p(\cdot), q(\cdot)}(\mathbb{R}^n) \right)' \subset \mathfrak{L}^{p'(\cdot),  q'(\cdot)}(\mathbb{R}^n)$ 
embeds continuously.

To complete the proof, we will now prove our claim. Let $f \in (L^{p(\cdot), q(\cdot)}(0, \infty))'$ such that  
$\mu_{p'(\cdot), q'(\cdot)}(f) < \infty$. Suppose that $\| f \|_{L^{p'(\cdot), q'(\cdot)}(0, \infty)} > 1$, then there exists 
$c > 1$ such that $\mu_{p'(\cdot), q'(\cdot)}(f/c) = 1$. Now, we define $h(t) = t^{q'(t)/p'(t) - 1} |f(t)/c|^{q'(t)-1}$, then 
\[
1 < c = c \, \mu_{p'(\cdot), q'(\cdot)}(f/c)  = \int_{0}^{\infty} |f(t)| h(t) dt,
\]
by \cite[Theorem 2.16]{Nekvinda} we have
\[
1 < c \leq \| f \|_{(L^{p(\cdot), q(\cdot)}(0, \infty))'} \| h \|_{L^{p(\cdot), q(\cdot)}(0, \infty)}.
\]
On the other hand, by taking into account that $\frac{1}{p(t)} - \frac{1}{q(t)} = \frac{1}{q'(t)} - \frac{1}{p'(t)}$, a computation gives
\[
\mu_{p(\cdot), q(\cdot)}(h) = \mu_{p'(\cdot), q'(\cdot)}(f/c) = 1.
\]
So $\| h \|_{L^{p(\cdot), q(\cdot)}(0, \infty)} \leq 1$, which implies that $1 < \| f \|_{(L^{p(\cdot), q(\cdot)}(0, \infty))'}$. This concludes the proof.
\end{proof}

\begin{proposition}
Let $p(\cdot)$ and $q(\cdot) \in \mathbb{P}_{1}(0, \infty)$ such that $q(0) \leq p(0)$ and $p(\infty) \leq q(\infty)$. Then 
$\left( \mathfrak{L}^{p(\cdot), q(\cdot)}(\mathbb{R}^n) \right)'$ is isomorphic to the dual space 
$\left( \mathfrak{L}^{p(\cdot), q(\cdot)}(\mathbb{R}^n) \right)^{*}$.
\end{proposition}

\begin{proof}
Follows from Proposition \ref{abs cont} and \cite[Proposition 3.15 (i)]{Lorist}.
\end{proof}

Let $f$ be a locally integrable function on $\mathbb{R}^{n}$. The Hardy-Littlewood maximal function of $f$ is defined by
\begin{equation} \label{max op}
Mf(x) = \sup_{Q \ni x} \frac{1}{|Q|} \int_{Q} |f(y)| dy,
\end{equation}
where the supremum is taken over all cubes $Q$ containing $x$.

The following two results state the boundedness of the maximal operator $M$ on $\mathfrak{L}^{p(\cdot), q(\cdot)}(\mathbb{R}^n)$ and the Fefferman-Stein vector-valued inequality for $M$. Both results are crucial to study the variable Hardy-Lorentz spaces associated with 
$\mathfrak{L}^{p(\cdot), q(\cdot)}(\mathbb{R}^n)$.

\begin{proposition} \label{acot MHL}
Let $p(\cdot)$ and $q(\cdot) \in \mathbb{P}_{1}(0, \infty)$. Then, the Hardy-Littlewood maximal operator $M$ is bounded on 
$\mathfrak{L}^{p(\cdot), q(\cdot)}(\mathbb{R}^n)$.
\end{proposition}

\begin{proof}
Follows from \cite[Theorem 4.4]{Zhao} with $\psi_j (x)= x$ for every $j= 1, ..., n$.
\end{proof}

\begin{proposition} \label{Feff-Stein 1}
Let $p(\cdot)$ and $q(\cdot) \in \mathbb{P}_{0}(0, \infty)$. Then, for any $1 < r < \infty$ and any $\tau \in (0, \min \{1, p_{-}, q_{-}\})$, there exists a positive constant $C$ such that for all sequences $\{ f_j \}_{j=1}^{\infty}$ of measurable functions
\[
\left\| \left( \sum_{j=1}^{\infty} (M f_j)^r \right)^{1/r} \right\|_{\mathfrak{L}^{p(\cdot)/\tau, \,\, q(\cdot)/\tau}(\mathbb{R}^n)} \leq C \left\| \left( \sum_{j=1}^{\infty} |f_j|^r \right)^{1/r} \right\|_{\mathfrak{L}^{p(\cdot)/\tau, \,\, q(\cdot)/\tau}(\mathbb{R}^n)}, 
\]
where $M$ is the Hardy-Littlewood maximal operator.
\end{proposition}

\begin{proof}
Fix $\tau \in (0, \min \{1, p_{-}, q_{-}\})$ and take $p_0 \in (\tau, \min \{1, p_{-}, q_{-}\})$, then $p(\cdot)/p_0$ and 
$q(\cdot)/p_0 \in \mathbb{P}_{1}(0, \infty)$. Then, by Theorem \ref{dual}, we have 
\[
\left( ( \mathfrak{L}^{p(\cdot), q(\cdot)}(\mathbb{R}^n) )^{\frac{1}{p_0}} \right)' = \left(\mathfrak{L}^{\frac{p(\cdot)}{p_0}, \frac{q(\cdot)}{p_0}}
(\mathbb{R}^n) \right)' = \mathfrak{L}^{\left(\frac{p(\cdot)}{p_0} \right)', \left(\frac{q(\cdot)}{p_0} \right)'}(\mathbb{R}^n),
\]
and by Proposition \ref{acot MHL}, $M$ is bounded on $\left( ( \mathfrak{L}^{p(\cdot), q(\cdot)}(\mathbb{R}^n) )^{\frac{1}{p_0}} \right)'$. Finally, we apply \cite[Proposition 22]{Rocha} with $u = r$ and $1/\sigma = \tau \in (0, p_0 )$.
\end{proof}

\section{Variable Hardy-Lorentz spaces}

We are now in a position to define the variable Hardy-Lorentz spaces associated with $\mathfrak{L}^{p(\cdot), q(\cdot)}(\mathbb{R}^n)$ following to \cite{Sawano}.

Let $M$ be the Hardy-Littlewood maximal operator given by (\ref{max op}). For a locally integrable function $f$ on $\mathbb{R}^{n}$ and
 $\theta \in (0, \infty)$, the \textit{powered Hardy-Littlewood maximal function} $M^{(\theta)}f$ is defined by
\[
(M^{(\theta)}f)(x) = \left[ M (|f|^{\theta})(x) \right]^{1/ \theta}.
\]

In \cite{Sawano}, the authors developed the theory of Hardy spaces associated with a ball quasi-Banach function space $(X, \| \cdot \|_X)$ under the following two assumptions:

\

$A1)$ Assume that, for some $\theta, s \in (0, 1]$ with $\theta < s$, there exists a positive constant $C$ such that for any sequence of functions $\{ f_j \}_{j=1}^{\infty} \subset L^{1}_{loc}(\mathbb{R}^n)$
\begin{equation} \label{A1}
\left\| \left\{ \sum_{j=1}^{\infty} (M^{(\theta)} f_j)^{s} \right\}^{1/s} \right\|_{X} \leq C
\left\| \left\{ \sum_{j=1}^{\infty} |f_j|^{s} \right\}^{1/s}  \right\|_{X}.
\end{equation}

$A2)$ Assume that there exist $s \in (0, 1]$, $\widetilde{q} \in (1, \infty]$ and a positive constant $C$ such that
for any $f \in (X^{1/s})'$
\begin{equation} \label{A2}
\left\|  M^{((\widetilde{q}/s)')} f \right\|_{(X^{1/s})'} \leq C \left\| f  \right\|_{(X^{1/s})'}.
\end{equation}

\begin{remark} \label{cond equiv A2}
We observe that (\ref{A2}) is equivalent to that the Hardy-Littlewood maximal operator $M$ be bounded on $[(X^{1/s})']^{1/(\widetilde{q}/s)'}$.
\end{remark}

The following theorem states that (\ref{A1}) and (\ref{A2}) hold true for $X = \mathfrak{L}^{p(\cdot), q(\cdot)}(\mathbb{R}^n)$, when 
$p(\cdot)$ and $q(\cdot) \in \mathbb{P}_{0}(0, \infty)$.

\begin{theorem} \label{Asump A1A2}
Let $p(\cdot)$ and $q(\cdot) \in \mathbb{P}_{0}(0, \infty)$ and $X = \mathfrak{L}^{p(\cdot), q(\cdot)}(\mathbb{R}^n)$. Then the assumption 
(\ref{A1}) holds true for any $s \in (0, 1]$ and any $\theta \in (0, \min\{ s, p_{-}, q_{-} \})$, and the assumption (\ref{A2}) holds true 
for any $s \in (0, \min\{ 1, p_{-}, q_{-} \})$ and any $\widetilde{q} \in (\max \{1, p_{+}, q_{+} \}, \infty]$.
\end{theorem}

\begin{proof}
Given $s \in (0, 1]$, let $\theta \in (0, \min\{ s, p_{-}, q_{-} \})$, then
\[
\left\Vert  \left\{ \sum_{j=1}^{\infty} (M^{(\theta)} f_j)^{s} \right\}^{1/s} \right\Vert_{X} =
\left\Vert  \left\{ \sum_{j=1}^{\infty} (M \vert f_j \vert^{\theta})^{s/\theta} \right\}^{1/s} \right\Vert_{X} =
\left\Vert  \left\{ \sum_{j=1}^{\infty} (M \vert f_j \vert^{\theta})^{s/\theta} \right\}^{\theta/s} \right\Vert_{X^{1/\theta}}^{1/\theta},
\]
applying Proposition \ref{Feff-Stein 1} with $\tau = \theta$, $r = s/\theta$ and $\vert f_j \vert^{\theta}$ instead of $f_j$, we get
\[
\leq C \left\Vert  \left\{ \sum_{j=1}^{\infty} (\vert f_j \vert^{\theta})^{s/\theta} \right\}^{\theta/s} \right\Vert_{X^{1/\theta}}^{1/\theta}
= C
\left\Vert \left\{ \sum_{j=1}^{\infty} |f_j|^{s} \right\}^{1/s}  \right\Vert_{X}.
\]
Thus, the assumption (\ref{A1}) holds true for $X = \mathfrak{L}^{p(\cdot), q(\cdot)}(\mathbb{R}^n)$, any $s \in (0, 1]$ and 
any $\theta \in (0, \min\{ s, p_{-}, q_{-} \})$.

Finally, by Theorem \ref{dual}, Proposition \ref{acot MHL} and Remark \ref{cond equiv A2}, the assumption (\ref{A2}) holds true for 
$X = \mathfrak{L}^{p(\cdot), q(\cdot)}(\mathbb{R}^n)$, any $s \in (0, \min\{ 1, p_{-}, q_{-} \})$ and any 
$\widetilde{q} \in (\max \{1, p_{+}, q_{+} \}, \infty]$.
\end{proof}

Given $L \in \mathbb{Z}_{+}$, let 
\[
\mathcal{F}_{L}=\left\{ \varphi \in \mathcal{S}(\mathbb{R}^{n}):\sum\limits_{\left\vert \mathbf{\beta }\right\vert \leq L}\sup\limits_{x\in \mathbb{R}^{n}}\left( 1+\left\vert x\right\vert \right)^{L}\left\vert \partial^{\mathbf{\beta }}
\varphi(x) \right\vert := \| \varphi \|_{\mathcal{S}(\mathbb{R}^{n}), \, L}  \leq 1\right\}.
\] 
Given $f \in \mathcal{S}'(\mathbb{R}^{n})$, define the maximal function $\mathcal{M}^{0}_{L} f$ by 
\begin{equation} \label{grand max cero}
\mathcal{M}^{0}_{L} f(x) :=\sup \left\{ |\left( \phi_t \ast f\right)(x) | : t > 0, \phi \in \mathcal{F}_{L} \right\}.
\end{equation}

Now, we introduce the Hardy type space associated with $\mathfrak{L}^{p(\cdot), q(\cdot)}(\mathbb{R}^n)$.

\begin{definition} \label{Hx space}
Let $p(\cdot)$, $q(\cdot)\in \mathbb{P}_{0}(0, \infty)$, $r \in (0, \min\{ 1, p_{-}, q_{-}\})$ and $L \geq \lfloor  n/r + 3 \rfloor$. Then 
the Hardy space $H^{p(\cdot), q(\cdot)}(\mathbb{R}^n)$ associated with $\mathfrak{L}^{p(\cdot), q(\cdot)}(\mathbb{R}^n)$ is defined as
\[
H^{p(\cdot), q(\cdot)}(\mathbb{R}^n) = \left\{ f \in \mathcal{S}'(\mathbb{R}^n) : \| \mathcal{M}^{0}_{L} f \|_{\mathfrak{L}^{p(\cdot), q(\cdot)}} < \infty \right\}.
\]
Fixed  $L \geq \lfloor  n/r + 3 \rfloor$, we consider $\| f \|_{H^{p(\cdot), q(\cdot)}} = \| \mathcal{M}^{0}_{L} f \|_{\mathfrak{L}^{p(\cdot), q(\cdot)}}$.
\end{definition}

\begin{remark}
For $r \in (0, \min\{ 1, p_{-}, q_{-}\})$, Proposition \ref{acot MHL} gives the boundedness of the Hardy-Littlewood maximal operator $M$ on $(\mathfrak{L}^{p(\cdot), q(\cdot)}(\mathbb{R}^n))^{1/r}$. Then, the well-definition of variable Hardy-Lorentz spaces follows from \cite[Theorem 3.1 - (ii)]{Sawano}.
\end{remark}

Before establishing the infinite atomic decomposition for elements of $H^{p(\cdot), q(\cdot)}(\mathbb{R}^n)$, we give the definition of 
$\mathfrak{L}^{p(\cdot), q(\cdot)}$-atom.

\begin{definition}
Let $p(\cdot)$ and $q(\cdot) \in \mathbb{P}_{0}(0, \infty)$ and let $r \in [1, \infty]$. Assume that $d \in \mathbb{Z}_{+}$ satisfies 
$d \geq d_{\mathfrak{L}^{p(\cdot), q(\cdot)}}$, where $d_{\mathfrak{L}^{p(\cdot), q(\cdot)}} := \lfloor n (1/\theta - 1) \rfloor$ and 
$\theta \in (0, 1)$ is as in Theorem \ref{Asump A1A2}. Then the function $a(\cdot)$ is called an 
$({\mathfrak{L}^{p(\cdot), q(\cdot)}}, r, d)$-atom if there exists a cube $Q \in \mathcal{Q}$ such that $\supp(a) \subset Q$,
\[
\| a \|_{L^r(\mathbb{R}^n)} \leq \frac{|Q|^{1/r}}{\| \chi_Q \|_{{\mathfrak{L}^{p(\cdot), q(\cdot)}}}},
\]
and $a(\cdot) \perp \mathcal{P}_d$.
\end{definition}

\begin{remark} \label{infinite atom}
For $r \geq 1$ fixed, every $({\mathfrak{L}^{p(\cdot), q(\cdot)}}, \infty, d)$-atom is an $({\mathfrak{L}^{p(\cdot), q(\cdot)}}, r, d)$-atom.
\end{remark}

The following infinite atomic decomposition for elements of $H^{p(\cdot), q(\cdot)}(\mathbb{R}^n) \cap L^{p_0}(\mathbb{R}^n)$ is a version of \cite[Theorem 3.7]{Sawano}.

\begin{theorem} \label{X atomic decomp}
Let $p(\cdot)$ and $q(\cdot) \in \mathbb{P}_{0}(0, \infty)$ and let $\theta, s \in (0,1]$ be as in Theorem \ref{Asump A1A2} satisfying 
(\ref{A1}), let $d \geq d_{\mathfrak{L}^{p(\cdot), q(\cdot)}}$ be a fixed integer and $p_0 \in (1, \infty)$. Then, for every $f \in H^{p(\cdot), q(\cdot)}(\mathbb{R}^n) \cap L^{p_0}(\mathbb{R}^n)$, there exist a sequence $\{ a_j \}_{j=1}^{\infty}$ of $(\mathfrak{L}^{p(\cdot), q(\cdot)}, \infty, d)$-atoms, supported on the cubes $\{ Q_j \}_{j=1}^{\infty} \subset \mathcal{Q}$, and a sequence 
$\{ \lambda_j \}_{j=1}^{\infty} \subset (0, \infty)$ such that
\begin{equation} \label{converg Lp}
f = \sum_{j=1}^{\infty} \lambda_j a_j \,\,\,\,\,\,\,\, \text{in} \,\,\,\, L^{p_0}(\mathbb{R}^n),
\end{equation}
and
\begin{equation} \label{atomic norm}
\left\| \left\{ \sum_{j=1}^{\infty} \left( \frac{\lambda_j}{\| \chi_{Q_j} \|_{\mathfrak{L}^{p(\cdot), q(\cdot)}}} \right)^s \chi_{Q_j} \right\}^{1/s} \right\|_{\mathfrak{L}^{p(\cdot), q(\cdot)}(\mathbb{R}^n)} \lesssim_{s}  \| f \|_{H^{p(\cdot), q(\cdot)}(\mathbb{R}^n)},
\end{equation}
where the implicit positive constant is independent of $f$, but depends on $s$.
\end{theorem}

\begin{proof}
The existence of a such atomic decomposition as well as the validity of inequality (\ref{atomic norm}) are guaranteed by 
\cite[Theorem 3.7]{Sawano}. In principle, the convergence in (\ref{converg Lp}) is in $\mathcal{S}'(\mathbb{R}^n)$. To see the 
convergence of the atomic series to $f$ in $L^{p_0}(\mathbb{R}^{n})$, we point out that the construction of a such atomic decomposition 
(see \cite[Proposition 4.3]{Sawano}) is analogous to the one given for classical Hardy spaces (see \cite[Chapter III]{Stein}). Since 
$f \in L^{p_0}(\mathbb{R}^{n})$ we have that $\mathcal{M}_{L}^{0} f \in L^{p_0}(\mathbb{R}^{n})$. So, following the proof of 
\cite[Theorem 3.1]{Pablo}, we obtain (\ref{converg Lp}).
\end{proof}

From Propositions \ref{norm Lpq}, \ref{s-convex} and \ref{abs cont}, \cite[Corollary 3.11]{Sawano}, and since \cite[Remark 3.12]{Sawano} also holds true for any $p_0 \in (1, \infty)$, we have the following result.

\begin{proposition} 
Let $p(\cdot)$ and $q(\cdot) \in \mathbb{P}_{0}(0, \infty)$ such that $q(0) \leq p(0)$ and $p(\infty) \leq q(\infty)$. Then, for any 
$p_0 \in (1, \infty]$, $H^{p(\cdot), q(\cdot)}(\mathbb{R}^n) \cap L^{p_0}(\mathbb{R}^n)$ is dense in $H^{p(\cdot), q(\cdot)}(\mathbb{R}^n)$.
\end{proposition}

Now, following to \cite{Yan}, we introduce the finite Hardy-Lorentz space with variable exponents 
$H^{p(\cdot), q(\cdot), \widetilde{q}, s}_{fin}(\mathbb{R}^n)$. For them, we need the following definition of atom.

\begin{definition} \label{atom}
Let $p(\cdot)$ and $q(\cdot) \in \mathbb{P}_{0}(0, \infty)$, and $s \in (0, \min\{ 1, p_{-}, q_{-} \})$. For such $s$, let 
$\theta \in (0,s)$ and $\widetilde{q} \in (1, \infty]$ be as in Theorem \ref{Asump A1A2}. Assume that 
$d \in \mathbb{Z}_{+}$ satisfies $d \geq d_{\mathfrak{L}^{p(\cdot), q(\cdot)}}$, where 
$d_{\mathfrak{L}^{p(\cdot), q(\cdot)}} := \lfloor n (1/\theta - 1) \rfloor$. Then the function $a(\cdot)$ is called an 
$(\mathfrak{L}^{p(\cdot), q(\cdot)}, \widetilde{q}, d)$-atom if there exists a ball $B$ such that $\supp(a) \subset B$,
\[
\| a \|_{L^{\widetilde{q}}(\mathbb{R}^n)} \leq \frac{|B|^{1/\widetilde{q}}}{\| \chi_B \|_{\mathfrak{L}^{p(\cdot), q(\cdot)}}},
\]
and $a(\cdot) \perp \mathcal{P}_d$.
\end{definition}

\begin{definition}
Let $p(\cdot)$, $q(\cdot)$, $\widetilde{q}$, $d$ and $s$ be as in Definition \ref{atom}. The finite atomic Hardy-Lorentz space $H^{p(\cdot), q(\cdot), \widetilde{q}, d}_{fin}(\mathbb{R}^n)$, associated to $\mathfrak{L}^{p(\cdot), q(\cdot)}(\mathbb{R}^n)$, is defined to be the set of all finite linear combinations of $({\mathfrak{L}^{p(\cdot), q(\cdot)}}, \widetilde{q}, d)$-atoms. The quasi-norm 
$\| \cdot \|_{H^{p(\cdot), q(\cdot), \widetilde{q}, d}_{fin}(\mathbb{R}^n)}$ in $H^{p(\cdot), q(\cdot), \widetilde{q}, d}_{fin}(\mathbb{R}^n)$ is defined, for any $f \in H^{p(\cdot), q(\cdot), \widetilde{q}, d}_{fin}(\mathbb{R}^n)$, by setting
\[
\| f \|_{H^{p(\cdot), q(\cdot), \widetilde{q}, d}_{fin}(\mathbb{R}^n)} := 
\]
\[
\inf \left\{ \left\| \left\{ \sum_{j=1}^{N} \left( \frac{\lambda_j}{\| \chi_{B_j} \|_{\mathfrak{L}^{p(\cdot), q(\cdot)}}} 
\right)^{s} \chi_{B_j} \right\}^{1/s} \right\|_{\mathfrak{L}^{p(\cdot), q(\cdot)}} : \\ f = \sum_{j=1}^{N} \lambda_j a_{j}, \, \{ \lambda_j \}_{j=1}^{N} \subset [0, \infty) \right\},
\]
where the infimum is taken over all finite linear combinations of $f$ in terms of $({\mathfrak{L}^{p(\cdot), q(\cdot)}}, \widetilde{q}, d)$ -atoms $\{ a_j \}_{j=1}^{N}$ supported, respectively, in the balls $\{ B_j \}_{j=1}^{N}$.
\end{definition}

\begin{theorem} (\cite[Theorem 1.10]{Yan}) \label{equiv quasi-norm}
Let $p(\cdot)$, $q(\cdot)$, $\widetilde{q}$, and $d$ be as in Definition \ref{atom}. If $\widetilde{q} \in (1, \infty)$, then 
$H^{p(\cdot), q(\cdot), \widetilde{q}, d}_{fin}(\mathbb{R}^n) \subset H^{p(\cdot), q(\cdot)}(\mathbb{R}^n)$. Moreover, 
$\| \cdot \|_{H^{p(\cdot), q(\cdot), \widetilde{q}, d}_{fin}(\mathbb{R}^n)}$ and $\| \cdot \|_{H^{p(\cdot), q(\cdot)}(\mathbb{R}^n)}$ are equivalent quasi-norms on the space $H^{p(\cdot), q(\cdot), \widetilde{q}, d}_{fin}(\mathbb{R}^n)$.
\end{theorem}

\begin{corollary}
Let $p(\cdot)$ and $q(\cdot) \in \mathbb{P}_{0}(0, \infty)$ such that $q(0) \leq p(0)$ and $p(\infty) \leq q(\infty)$, and let $s$, 
$\widetilde{q}$ and $d$ be as in Definition \ref{atom}. Then 
$H^{p(\cdot), q(\cdot), \widetilde{q}, d}_{fin}(\mathbb{R}^n) \subset H^{p(\cdot), q(\cdot)}(\mathbb{R}^n)$ densely. 
\end{corollary}

\begin{proof} By Propositions \ref{norm Lpq}, \ref{s-convex} and \ref{abs cont}, the corollary follows from 
\cite[Theorems 3.6 and 3.7]{Sawano} (which also hold with balls instead of cubes), and \cite[Corollary 3.11 - (ii)]{Sawano}.
\end{proof}

\section{Estimates for classical operators}

In this section, we apply \cite[Theorem 28 and 29]{Rocha} and \cite[Theorem 31]{Rocha2} to establish the boundedness of singular 
integrals and fractional type operators in $H^{p(\cdot), q(\cdot)}(\mathbb{R}^n)$. By means of \cite[Proposition 21]{Rocha}, 
we also provide a Fefferman-Stein vector-valued inequality for the fractional maximal operator on the $r$-convexification of 
$\mathfrak{L}^{p(\cdot), q(\cdot)}(\mathbb{R}^n)$.

\subsection{Singular integrals} Let $\Omega \in C^{\infty}(S^{n-1})$ with $\int_{S^{n-1}} \Omega(u) d\sigma(u)=0$. We define the operator $T$ by
\begin{equation} \label{T sing}
Tf(x) = \lim_{\epsilon \rightarrow 0^{+}} \int_{|y| > \epsilon} \frac{\Omega(y/|y|)}{|y|^{n}} f(x-y) \, dy, \,\,\,\, x \in \mathbb{R}^{n}. 
\end{equation}
It is well known that $\widehat{Tf}(\xi) = m(\xi) \widehat{f}(\xi)$, where the multiplier $m$ is homogeneous of degree $0$ and is indefinitely diferentiable on $\mathbb{R}^{n} \setminus \{0\}$. Moreover, if $k(y) = \frac{\Omega(y/|y|)}{|y|^{n}}$ we have
\begin{equation} \label{k estimate}
|\partial^{\beta} k(y) |\leq C |y|^{-n-|\beta|}, \,\,\,\,  \textit{for all } \,\, y \neq 0 \,\, \textit{and all multi-index} \,\, \beta. 
\end{equation}
The operator $T$ results bounded on $L^{p}(\mathbb{R}^{n})$ for every $1 < p < +\infty$ (see \cite{Elias}).

\begin{theorem}  \label{Sing estimates}
Let $p(\cdot)$ and $q(\cdot) \in \mathbb{P}_{0}(0, \infty)$ such that $q(0) \leq p(0)$ and $p(\infty) \leq q(\infty)$.
Then the singular operator $T$ given by (\ref{T sing}) extends to a bounded operator 
$H^{p(\cdot), q(\cdot)}(\mathbb{R}^n) \to \mathfrak{L}^{p(\cdot), q(\cdot)}(\mathbb{R}^n)$ and 
$H^{p(\cdot), q(\cdot)}(\mathbb{R}^n) \to H^{p(\cdot), q(\cdot)}(\mathbb{R}^n)$.
\end{theorem}

\begin{proof}
By Theorem \ref{Asump A1A2}, Propositions \ref{norm Lpq} and \ref{s-convex}, Propostition \ref{abs cont}, and since $M$ is bounded on 
$((\mathfrak{L}^{p(\cdot), q(\cdot)}(\mathbb{R}^n))^{1/p_0})'$ for any $p_0 \in (0, \min\{ 1, p_{-}, q_{-}\})$ (see Theorem \ref{dual} and Proposition \ref{acot MHL}), we can apply \cite[Theorem 28]{Rocha} and the theorem follows.
\end{proof}

\subsection{Fractional maximal operator} For $0 < \alpha < n$, we define the \textit{fractional maximal operator} $M_{\alpha}$ by
\[
(M_{\alpha}f)(x) = \sup_{Q \ni x} |Q|^{\frac{\alpha}{n} - 1}\int_{Q} |f(y)| \, dy,
\]
where $f$ is a locally integrable function on $\mathbb{R}^{n}$ and the supremum is taken over all cubes $Q$ containing $x$.

\begin{proposition} \label{vector-valued fract}
Let $0 < \alpha < n$, $r \in (1, \infty)$, $p(\cdot)$ and $q(\cdot) \in \mathbb{P}_{0}(0, \infty)$ such that 
$0 < p_{+}, \,\, q_{+} < \frac{n}{\alpha}$. If $p_0 \in (0, \min\{ \frac{n}{n+\alpha}, p_{-}, q_{-}\})$, 
$\frac{1}{u(t)} : = \frac{1}{p(t)} - \frac{\alpha}{n}$ and $\frac{1}{v(t)} : = \frac{1}{q(t)} - \frac{\alpha}{n}$, then for 
any $\sigma > 1/p_0$ and any sequence of measurable functions $\{ f_j \}_{j=1}^{\infty}$,
\begin{equation} \label{feff-stein fract max}
\left\| \left\{ \sum_{j=1}^{\infty} (M_{\frac{\alpha}{\sigma}} f_j)^{r} \right\}^{1/r} \right\|_{\mathfrak{L}^{\sigma u(\cdot), 
\sigma v(\cdot)}(\mathbb{R}^n)} \lesssim \left\| \left\{ \sum_{j=1}^{\infty} |f_j|^{r} \right\}^{1/r} \right\|_{\mathfrak{L}^{\sigma p(\cdot), 
\sigma q(\cdot)}(\mathbb{R}^n)}.
\end{equation}
\end{proposition}

\begin{proof}
Let $0 < \alpha < n$, for any $p_0 \in (0, \min\{ \frac{n}{n+\alpha}, p_{-}, q_{-}\})$ we put 
$\frac{1}{q_0} := \frac{1}{p_0} - \frac{\alpha}{n}$, then $0 < p_0 \leq q_0 \leq 1$ 
and $q_0 \in (0, \min\{1, u_{-}, v_{-}\})$. Applying Theorem \ref{dual}, we obtain
\begin{equation} \label{XYp0q0}
\left((\mathfrak{L}^{u(\cdot), v(\cdot)}(\mathbb{R}^n))^{1/q_0} \right)' = 
\left[\left((\mathfrak{L}^{p(\cdot), q(\cdot)}(\mathbb{R}^n))^{1/p_0} \right)' \right]^{p_0/q_0}.
\end{equation}
On the other hand, by Proposition \ref{acot MHL}, the Hardy-Littlewood maximal operator $M$ is bounded on 
$\left((\mathfrak{L}^{u(\cdot), v(\cdot)}(\mathbb{R}^n))^{1/q_0} \right)'$. Then, (\ref{feff-stein fract max}) follows from 
\cite[Proposition 21]{Rocha}.
\end{proof}

\subsection{Riesz potentials} Given $0 < \alpha < n$, the Riesz potential $I_{\alpha}$ on $\mathbb{R}^{n}$ is defined by
\begin{equation} \label{Riesz}
(I_{\alpha}f)(x) = \int_{\mathbb{R}^{n}} |x-y|^{\alpha - n} f(y) \, dy.
\end{equation}
It is clear that $K_{\alpha}(\cdot) := |\cdot|^{\alpha - n} \in C^{\infty}(\mathbb{R}^{n} \setminus \{0\}) = \displaystyle{\bigcap_{N \in \mathbb{N}}} C^{N}(\mathbb{R}^{n} \setminus \{0\})$ and for every $N \in \mathbb{N}$ satisfies
\[
\left|(\partial^{\beta}K_{\alpha})(x) \right| \lesssim |x|^{\alpha - n - |\beta|} \,\,\,\, \text{for all} \,\, |\beta| \leq N \,\,\, \text{and all} \,\, x \neq 0.
\]
Then, Theorem \ref{Asump A1A2}, Propositions \ref{norm Lpq} and \ref{s-convex}, Proposition \ref{abs cont}, (\ref{XYp0q0}) and Proposition \ref{acot MHL} together with 
\cite[Theorem 29]{Rocha} allow us to obtain the following result.

\begin{theorem}  \label{Riesz estimates}
Let $0 < \alpha < n$, $p(\cdot)$ and $q(\cdot) \in \mathbb{P}_{0}(0, \infty)$ such that $0 < p_{+}, \,\, q_{+} < \frac{n}{\alpha}$, 
$q(0) \leq p(0)$ and $p(\infty) \leq q(\infty)$. 
Then, for $\frac{1}{u(t)} : = \frac{1}{p(t)} - \frac{\alpha}{n}$ and $\frac{1}{v(t)} : = \frac{1}{q(t)} - \frac{\alpha}{n}$, 
the Riesz potential $I_{\alpha}$ given by (\ref{Riesz}) extends to a bounded operator 
$H^{p(\cdot), q(\cdot)}(\mathbb{R}^n) \to \mathfrak{L}^{u(\cdot), v(\cdot)}(\mathbb{R}^n)$ and 
$H^{p(\cdot), q(\cdot)}(\mathbb{R}^n) \to H^{u(\cdot), v(\cdot)}(\mathbb{R}^n)$.
\end{theorem}

\subsection{Fractional type integrals} Let $0 \leq \alpha <n$ and $m \in \mathbb{N} \cap \left(1 - \frac{\alpha}{n}, +\infty \right)$, we consider the following generalization of the Riesz potential
\begin{equation} \label{T}
T_{\alpha, m}f(x)=\int_{\mathbb{R}^{n}} \left\vert x-A_{1}y\right\vert ^{-\alpha
_{1}}...\left\vert x-A_{m}y\right\vert ^{-\alpha _{m}}f(y)dy, 
\end{equation}
where $\alpha _{1}+...+\alpha _{m}=n-\alpha$, and the $A_j$'s are $n \times n$ orthogonal matrices. For the case $\alpha = 0$, we assume 
that $A_i - A_j$ is invertible if $i \neq j$. The family of operators in (\ref{T}) contains to the Riesz potential. Indeed, taking $0 < \alpha < n$, $m=1$ and $A_1 = I$ in (\ref{T}), we have that $I_\alpha = T_{\alpha, 1}$. Our last result is contained in the following theorem.

\begin{theorem}
Let $0 \leq \alpha < n$, $m \in \mathbb{N} \cap \left(1 - \frac{\alpha}{n}, +\infty \right)$, $p(\cdot)$ and 
$q(\cdot) \in \mathbb{P}_{0}(0, \infty)$ such that $0 < p_{+}, \,\, q_{+} < \frac{n}{\alpha}$, $q(0) \leq p(0)$ and 
$p(\infty) \leq q(\infty)$. Then, for $\frac{1}{u(t)} : = \frac{1}{p(t)} - \frac{\alpha}{n}$ and 
$\frac{1}{v(t)} : = \frac{1}{q(t)} - \frac{\alpha}{n}$, the operator $T_{\alpha, m}$ given by (\ref{T}) extends to a bounded operator 
$H^{p(\cdot), q(\cdot)}(\mathbb{R}^n) \to \mathfrak{L}^{u(\cdot), v(\cdot)}(\mathbb{R}^n)$.
\end{theorem}

\begin{proof}
From Definition \ref{var Lorentz}, it is clear that $\| \cdot \|_{\mathfrak{L}^{p(\cdot), q(\cdot)}}$ 
is $\mathcal{O}(n)$-invariant (i.e.: for every $f \in \mathfrak{L}^{p(\cdot), q(\cdot)}(\mathbb{R}^n)$ and $A \in \mathcal{O}(n)$, 
$f_A \in \mathfrak{L}^{p(\cdot), q(\cdot)}(\mathbb{R}^n)$ with 
$\| f_A \|_{\mathfrak{L}^{p(\cdot), q(\cdot)}} = \| f \|_{\mathfrak{L}^{p(\cdot), q(\cdot)}}$). Now, we set 
$X:= \mathfrak{L}^{p(\cdot), q(\cdot)}(\mathbb{R}^n)$ and $Y:= \mathfrak{L}^{u(\cdot), v(\cdot)}(\mathbb{R}^n)$. 
Let $0 \leq \alpha < n$, for any $p_0 \in (0, \min\{ \frac{n}{n+\alpha}, p_{-}, q_{-} \})$, we put 
$\frac{1}{q_0}:= \frac{1}{p_0} - \frac{\alpha}{n}$, then by (\ref{XYp0q0}) we have that $(Y^{1/q_0})' = ((X^{1/p_0})')^{p_0/q_0}$, and 
by Proposition \ref{acot MHL} and Theorem \ref{dual}, $M$ is bounded on 
$(Y^{1/q_0})'$. On the other hand, the constant $\widetilde{q}$ in (\ref{A2}) can be chosen such that 
$\widetilde{q} > \max \left\{1, \frac{p_0 n}{\alpha} \right\}$ or $\widetilde{q} > \max \left\{1, p_0 \left(1+2^{n+3} \| M \|_{(X^{1/p_0})'} \right) \right\}$, according to the case $0 < \alpha < n$ or $\alpha = 0$.

Finally, the theorem follows from Theorem \ref{Asump A1A2}, Propositions \ref{norm Lpq} and \ref{s-convex}, Proposition \ref{abs cont} and \cite[Theorem 31]{Rocha2}.
\end{proof}

Pablo Rocha, Instituto de Matem\'atica (INMABB), Departamento de Matem\'atica, Universidad Nacional del Sur (UNS)-CONICET, Bah\'ia Blanca, Argentina. \\
{\it e-mail:} pablo.rocha@uns.edu.ar

\end{document}